\documentclass[11pt,reqno,a4paper]{amsart}
\usepackage{amsmath,amsthm,amssymb,amsfonts,xypic,graphics,color}  
\usepackage[english]{babel}
\usepackage[mathcal]{eucal}

\usepackage{pdfsync}
\allowdisplaybreaks
\usepackage[all]{xy}

\newcommand{\aaa}{\ensuremath{\mathcal{A}}}

\newcommand{\rrr}{\ensuremath{\mathcal{R}}}

\newcommand{\ttt}{\ensuremath{\mathcal{T}}}

\DeclareMathOperator*{\myinf}{in\vphantom{p}f}

\begin{document}

\title[A quantitative Lindeberg CLT for random vectors]{Stein's method and a quantitative Lindeberg CLT for the Fourier transforms of random vectors}
\author{B. Berckmoes \and R. Lowen \and J. Van Casteren}
\date{}

\thanks{(1) Ben Berckmoes is PhD fellow at the Fund for Scientific Research of Flanders (FWO)}
\thanks{(2) Corresponding author: R. Lowen, email address: bob.lowen@uantwerpen.be}

\maketitle

\newtheorem{pro}{Proposition}[section]
\newtheorem{lem}[pro]{Lemma}
\newtheorem{thm}[pro]{Theorem}
\newtheorem{de}[pro]{Definition}
\newtheorem{co}[pro]{Comment}
\newtheorem{no}[pro]{Notation}
\newtheorem{vb}[pro]{Example}
\newtheorem{vbn}[pro]{Examples}
\newtheorem{gev}[pro]{Corollary}
\newtheorem{vrg}[pro]{Question}

\newtheorem{proA}{Proposition}
\newtheorem{lemA}[proA]{Lemma}
\newtheorem{thmA}[proA]{Theorem}
\newtheorem{deA}[proA]{Definition}
\newtheorem{coA}[proA]{Comment}
\newtheorem{noA}[proA]{Notation}
\newtheorem{vbA}[proA]{Example}
\newtheorem{vbnA}[proA]{Examples}
\newtheorem{gevA}[proA]{Corollary}
\newtheorem{vrgA}[proA]{Question}

\newtheorem{proB}{Proposition}
\newtheorem{lemB}[proB]{Lemma}
\newtheorem{thmB}[proB]{Theorem}
\newtheorem{deB}[proB]{Definition}
\newtheorem{coB}[proB]{Comment}
\newtheorem{noB}[proB]{Notation}
\newtheorem{vbB}[proB]{Example}
\newtheorem{vbnB}[proB]{Examples}
\newtheorem{gevB}[proB]{Corollary}
\newtheorem{vrgB}[proB]{Question}

\newcommand{\Lin}{\textrm{\upshape{Lin}}\left(\left\{\xi_{n,k}\right\}\right)}
\newcommand{\Eta}{\textrm{\upshape{H}}}
\newcommand{\Zeta}{\textrm{\upshape{Z}}}

\hyphenation{frame-work}
\hyphenation{dif-fe-rent}
\hyphenation{a-vai-la-ble}
\hyphenation{me-tric}
\hyphenation{to-po-lo-gi-cal}
\hyphenation{con-ti--nu-ous-ly}
\hyphenation{de-pen-ding}
\hyphenation{ne-gli-gi-ble}
\hyphenation{de-ri-va-tive}

\begin{abstract}
We use a multivariate version of Stein's method to establish a quantitative Lindeberg CLT for the Fourier transforms of random $N$-vectors. We achieve this, conceptually mainly by constructing a natural approach structure on $N$-random vectors overlying the topology of weak convergence, and technically mainly by deducing a specific integral representation for the Hessian matrix of a solution to the Stein equation with test function $e_t(x) = \exp\left(- i \sum_{k=1}^N t_k x_k\right)$, where $t,x \in \mathbb{R}^N$.
\end{abstract}

 \section{Introduction and preliminaries}
 
 Before we start with the actual material of the paper we of course have to explain precisely what we mean when we say a ``quantitative Lindeberg CLT''. A classical CLT is a limit theorem for the weak topology, and as such it is a deterministic result saying that, under given conditions, a certain sequence converges weakly to a certain limit point. Classically, what happens if the required conditions are not met is very simple: namely, we do not know that there will be convergence, and if the given conditions are both necessary and sufficient we know that there definitely will not be convergence, and the matter ends there. However, there are situations wherein the deviation from the given conditions can in a very natural way be  measured numerically, in such a way that the smaller the measure of deviation is the better the conditions are approximated. This still will not help us to say anything more about the (topological) limit process and this is where index analysis in the context of approach theory comes into play (\cite{L97}, \cite{L15}). We replace the weak topology by a canonical ``weak approach structure'' and apply the full machinery of approach theory. Where a topological space allows for a notion of convergence, in exactly the same structural and canonical way, an approach space allows for a notion of limit operator, whereby for any sequence and any point the limit operator gives a numerical value indicating ``how far the point is away from being a limit point'' of the sequence, an ``index of convergence''. Of course, in order to be meaningful all this has to be such that there are natural relations between, in our case, on the one hand both the approach structure replacing the weak topology and the ``measure of deviation from conditions'' and on the other hand both the weak topology and the conditions. The following diagram makes things more precise. Herein $(\mathbb{R}^N, d_E)$ stands for the $N$-dimensional Euclidean space and metric, $(\rrr, \lambda)$ stands for the derived space of $\mathbb{R}^N$-valued random variables defined on some probability space $(\Omega, \aaa, \mathbb{P})$ with some, yet to be defined, approach limit operator $\lambda$. The usual way would be to go straight to the topological space $(\rrr, \ttt_w)$, but instead of doing this, we equip it with an approach structure which is such that (1) the underlying topology is the weak topology $\ttt_w$ and (2) it is obtained in a natural canonical way, as is the weak topology. Then the smaller the value of the limit operator (the index of convergence) is, the more convergence in the weak topology is being ``approximated'', or the better a ``virtual limit point'' approximates being a real limit point in the weak topology.
 $$
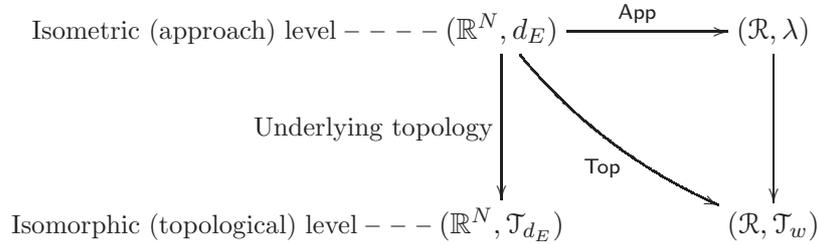
\xy
\xymatrix{ 
\text{\small{Isometric (approach) level}} \ar@{--}[r] & (\mathbb{R}^N, d_E) \ar@/_10pt/[rrdd]_{\mathsf{Top}} \ar[dd]_{\text{\small{Underlying topology}}} \ar[rr]^{\mathsf{App}}  &&  (\rrr, \lambda) \ar[dd] \\
&&&\\
\text{\small{Isomorphic (topological) level}} \ar@{--}[r] & (\mathbb{R}^N, \ttt_{d_E}) && (\rrr, \ttt_w)
}
\endxy
$$

The precise definition of the approach structure which we use is given in the second section.

The following simple example illustrates the meaning of limit operator and also serves to make clear to the reader that the notion of a limit operator, the index of convergence and the inequalities proven in the sequel are totally different from concepts and formulas related to the notion of rate or speed of convergence. Suppose we consider $\mathbb{R}$ not with its usual topology but with its usual metric. This in fact is a (metric) approach space and from the general theory it follows that the associated  limit operator takes the form
$$
\lambda(x_n \rightarrow x) = \limsup_{n \rightarrow \infty} |x_n - x|,
$$
which is a well-known expression in approximation theory (see e.g.~\cite{E72}, \cite{AMS82}). For instance if we take as sequence $x_n := (-1)^n\varepsilon$ for an arbitrary strictly positive $\varepsilon$ then the formula yields $\lambda(x_n \rightarrow x) = |x| + \varepsilon$ and thus in this case $0$ is the point which best approximates being a limit of this non-convergent sequence, with index of convergence equal to $\varepsilon$.

\vspace{10pt}

Now let us turn to the actual content of the present paper.
Let $\xi$ be a standard normally distributed random variable and $\{\xi_{n,k}\}$ a {\em 1-dimensional standard triangular array} (1-STA), i.e. a triangular array of real random variables 
\begin{equation*}
\begin{array}{cccc} 
\xi_{1,1} &  &  \\
\xi_{2,1} & \xi_{2,2} & \\
\xi_{3,1} & \xi_{3,2} & \xi_{3,3} \\
 & \vdots &
\end{array}
 \end{equation*}
with the following properties.
\begin{enumerate}
\item $\displaystyle{\forall n : \xi_{n,1}, \ldots, \xi_{n,n} \textrm{ are independent.}}$\\
\item $\displaystyle{\forall n, k : \mathbb{E}\left[\xi_{n,k}\right] = 0.}$\\
\item $\displaystyle{\forall n : \mathbb{E}\left[S_n^2\right] = 1}$ with $\displaystyle{S_n = \sum_{k=1}^n \xi_{n,k}.}$
\end{enumerate}

The Lindeberg CLT (\cite{F71}) provides a useful condition under which the rowwise sums of $\{\xi_{n,k}\}$ are asymptotically normally distributed. As usual, $\stackrel{w}{\rightarrow}$ stands for weak convergence.

\begin{thm}\upshape{(Lindeberg CLT)}\label{LCLT}
{\em Suppose that $\left\{\xi_{n,k}\right\}$ satisfies} Lindeberg's condition {\em in the sense that}
\begin{equation*}
\forall \epsilon > 0 : \sum_{k=1}^n \mathbb{E}\left[\xi_{n,k}^2 ; \left|\xi_{n,k}\right| > \epsilon\right] \rightarrow 0.\label{LindebergCondition}
\end{equation*}
{\em Then}
\begin{equation*}
S_n \stackrel{w}{\rightarrow} \xi.
\end{equation*}		
\end{thm}

Recall that the {\em Kolmogorov distance} $K$ between random variables $\eta$ and $\eta^\prime$ is defined as
\begin{equation*}
\sup_{x \in \mathbb{R}} \left|F_\eta(x) - F_{\eta^\prime}(x)\right|
\end{equation*}
where
\begin{equation*}
F_{\zeta}(x) = \mathbb{E}\left[1_{\left]-\infty,x\right]}\left(\zeta\right)\right] = \mathbb{P}\left[\zeta \leq x\right]
\end{equation*}
represents the cumulative distribution function of the random variable $\zeta$. It is well known that $K$ metrizes weak convergence to a continuously distributed random variable.

The following powerful result was obtained by Feller in \cite{F}.

\begin{thm}\label{thm:BHall}
There exists a universal constant $C > 0$ such that
\begin{align}
&\sup_{x \in \mathbb{R}} \left|F_{\xi}(x) - F_{S_n}(x)\right|\nonumber\\
&\leq C \left(\sum_{k=1}^n \mathbb{E}\left[\xi_{n,k}^2 ; \left|\xi_{n,k}\right| > 1\right] +\sum_{k=1}^n \mathbb{E}\left[\left|\xi_{n,k}\right|^3 ; \left|\xi_{n,k}\right| \leq 1\right]\right).\label{LocalUpperBoundIneqBHall}
\end{align}
\end{thm}

It was shown in \cite{F} that the constant $C$ in (\ref{LocalUpperBoundIneqBHall}) can be taken equal to 6. The first proof of (\ref{LocalUpperBoundIneqBHall}) based on Stein's method was given by Barbour and Hall in \cite{BH}. More recently, the result was improved by Chen and Shao in \cite{CS}, where it was shown that $C$ can be taken equal to 4.1. The proof in \cite{CS} is based on Chen's concentration inequality approach in combination with Stein's method.

Theorem \ref{thm:BHall} has two important consequences.

The first is immediate. It is known as the Berry-Esseen inequality.

\begin{thm}\upshape{(Berry-Esseen inequality)}
{\em There exists a universal constant $C > 0$ such that}
\begin{equation}
\sup_{x \in \mathbb{R}} \left|F_{\xi}(x) - F_{S_n}(x)\right| \leq C \sum_{k=1}^n \mathbb{E}\left[\left|\xi_{n,k}\right|^3\right] .\label{BerryEsseenIneq}
\end{equation}
\end{thm}

It was shown by Shevtsova in \cite{Sh10} that the constant $C$ in (\ref{BerryEsseenIneq}) can be taken equal to 0.56.

For the second consequence, we recall that it was pointed out by Loh in \cite{L} that the truncation at $1$ in (\ref{LocalUpperBoundIneqBHall}) is optimal in the sense that
\begin{align}
& \sum_{k=1}^n \mathbb{E}\left[\xi_{n,k}^2 ; \left|\xi_{n,k}\right| > 1\right] + \sum_{k=1}^n \mathbb{E}\left[\left|\xi_{n,k}\right|^3 ; \left|\xi_{n,k}\right| \leq 1\right]\nonumber\\
& \leq \myinf_{A}  \left(\sum_{k=1}^n \mathbb{E}\left[\xi_{n,k}^2 ; A\right] + \sum_{k=1}^n \mathbb{E}\left[\left|\xi_{n,k}\right|^3 ; \mathbb{R} \setminus A\right]\right),\label{LocalBoundOptimal}
\end{align}
the infimum being taken over all Borel subsets $A$ of the real line. Thus, applying (\ref{LocalUpperBoundIneqBHall}) and (\ref{LocalBoundOptimal}), we get, for each $\epsilon > 0$,
\begin{align*}
&\sup_{x \in \mathbb{R}} \left|F_{\xi}(x) - F_{S_n}(x)\right|\\
&\leq C \left(\sum_{k=1}^n \mathbb{E}\left[\xi_{n,k}^2 ; \left|\xi_{n,k}\right| > 1 \right] + \sum_{k=1}^n \mathbb{E}\left[\left|\xi_{n,k}\right|^3 ; \left|\xi_{n,k}\right| \leq 1\right]\right)\\
&\leq C \left(\sum_{k=1}^n \mathbb{E}\left[\xi_{n,k}^2 ; \left|\xi_{n,k}\right| > \epsilon \right] + \sum_{k=1}^n \mathbb{E}\left[\left|\xi_{n,k}\right|^3 ; \left|\xi_{n,k}\right| \leq \epsilon\right]\right)\\
&\leq C \left(\sum_{k=1}^n \mathbb{E}\left[\xi_{n,k}^2 ; \left|\xi_{n,k}\right| > \epsilon \right] + \epsilon \sum_{k=1}^n \mathbb{E}\left[\left|\xi_{n,k}\right|^2\right]\right)\\
&= C \left(\sum_{k=1}^n \mathbb{E}\left[\xi_{n,k}^2 ; \left|\xi_{n,k}\right| > \epsilon \right] + \epsilon\right)
\end{align*}
which, after calculating the superior limit of both sides and letting $\epsilon \downarrow 0$, yields
\begin{align}
&\limsup_{n \rightarrow \infty} \sup_{x \in \mathbb{R}} \left|F_{\xi}(x) - F_{S_n}(x)\right|\nonumber\\
&\leq C \sup_{\epsilon > 0} \limsup_{n \rightarrow \infty} \sum_{k=1}^n \mathbb{E}\left[\xi_{n,k}^2 ; \left|\xi_{n,k}\right| > \epsilon \right].\label{Ineq2Cor}
\end{align}

Inspired by (\ref{Ineq2Cor}), the {\em Lindeberg index} of $\left\{\xi_{n,k}\right\}$ was introduced by the authors in \cite{BLV13} as
\begin{equation*}
\textrm{\upshape{Lin}}\left(\left\{\xi_{n,k}\right\}\right) = \sup_{\epsilon > 0} \limsup_{n \rightarrow \infty} \sum_{k=1}^n \mathbb{E}\left[\xi_{n,k}^2 ; \left|\xi_{n,k}\right| > \epsilon \right].\label{LindebergIndex}
\end{equation*}
It is clear that $0 \leq \textrm{\upshape{Lin}}\left(\left\{\xi_{n,k}\right\}\right)  \leq 1$ and that $\left\{\xi_{n,k}\right\}$ satisfies Lindeberg's condition if and only if $\textrm{\upshape{Lin}}\left(\left\{\xi_{n,k}\right\}\right) = 0$. 

The following example, taken from \cite{BLV13}, provides some insight into how the Lindeberg index behaves.

Let $0 < \alpha < 1$, $\beta =  \frac{\alpha}{1 - \alpha}$
and set
\begin{equation}
s_n^2 = (1 + \beta) n - \beta \sum_{k=1}^n k^{-1} = n + \beta \sum_{k=1}^{n} \left(1 - k^{-1}\right).\label{deta0}
\end{equation}
Notice that $s_n^2 \rightarrow \infty$. Now consider the 1-STA $\left\{\eta_{\alpha,n,k}\right\}$ such that 
\begin{equation}
\mathbb{P}\left[\eta_{\alpha,n,k} = -1/s_n\right] = \mathbb{P}\left[\eta_{\alpha,n,k} = 1/s_n\right] = \frac{1}{2}\left(1 - \beta k^{-1}\right)\label{deta1}
\end{equation}
and
\begin{equation}
\mathbb{P}\left[\eta_{\alpha,n,k} = -\sqrt{k}/{s_n}\right] = \mathbb{P}\left[\eta_{\alpha,n,k} = \sqrt{k}/{s_n}\right] = \frac{1}{2}\beta k^{-1}.\label{deta2}
\end{equation} 
Then it was shown in \cite{BLV13} (Proposition 2.2) that
\begin{equation*}
\textrm{\upshape{Lin}}\left(\left\{\eta_{\alpha,n,k}\right\}\right) = \alpha
\end{equation*}
and that $\left\{\eta_{\alpha,n,k}\right\}$ is {\em infinitesimal} in the sense that
\begin{equation*}
\forall \epsilon > 0 : \max_{k=1}^n \mathbb{P}\left[\left|\eta_{\alpha,n,k}\right| > \epsilon\right] \rightarrow 0.
\end{equation*}

Now, as a second consequence of Theorem \ref{thm:BHall}, the following quantitative version of the Lindeberg CLT is yielded by (\ref{Ineq2Cor}).

\begin{thm}\upshape{(Quantitative Lindeberg CLT)}\label{thm:BerckIneq}
{\em There exists a universal constant $C > 0$ such that}
\begin{equation}
\limsup_{n \rightarrow \infty} \sup_{x \in \mathbb{R}} \left|F_{\xi}(x) - F_{S_n}(x)\right| \leq C \textrm{\upshape{Lin}}\left(\left\{\xi_{n,k}\right\}\right).\label{BerckIneq}
\end{equation} 
\end{thm}

Using an asymptotic smoothing technique and Stein's method, it was shown in \cite{BLV13} that under the mild assumption that $\left\{\xi_{n,k}\right\}$ be infinitesimal, the constant $C$ in (\ref{BerckIneq}) can be taken equal to $1$. Moreover in \cite{BLV13} it was also shown that the expression on the left side is actually an index of convergence for a natural approach structure and thus, if we denote the associated limit operator simply $\lambda$  the result reads

$$
\lambda ( S_n \rightarrow \xi) \leq  \textrm{\upshape{Lin}}\left(\left\{\xi_{n,k}\right\}\right).
$$

This was the situation in the one-dimensional case. We now turn to the multivariate case and see how the techniques and inequalities in the previous discussion can be extended. Throughout, we keep $N \in \mathbb{N}_0$ fixed and we let $\left|\cdot\right|$ stand for the norm and $\left\langle\cdot,\cdot\right\rangle$ for the inner product in Euclidean $N$-space $\mathbb{R}^N$. By a {\em random $N$-vector} we mean an $\mathbb{R}^N$-valued random variable. Furthermore, $\Xi$ is a standard normally distributed random $N$-vector and $\left\{\Xi_{n,k}\right\}$ an {\em $N$-dimensional standard triangular array} ($N$-STA), i.e. a triangular array of random $N$-vectors
\begin{equation*}
\begin{array}{cccc} 
\Xi_{1,1} &  &  \\
\Xi_{2,1} & \Xi_{2,2} & \\
\Xi_{3,1} & \Xi_{3,2} & \Xi_{3,3} \\
 & \vdots &
\end{array}
 \end{equation*}
with the following properties.
\begin{enumerate}
\item $\displaystyle{\forall n : \Xi_{n,1}, \ldots, \Xi_{n,n} \textrm{ are independent.}}$\\
\item $\displaystyle{\forall n, k : \mathbb{E}\left[\Xi_{n,k}\right] = 0.}$\\
\item $\displaystyle{\forall n : \textrm{\upshape{cov}}\left(\Sigma_n\right) = I_{N \times N}}$ with $\displaystyle{\Sigma_n = \sum_{k=1}^n \Xi_{n,k}}$.
\end{enumerate}
Note that the notion of $N$-STA coincides with the earlier introduced notion of 1-STA in the case where $N = 1$.

The Lindeberg CLT is now extended as follows (\cite{S11}).

\begin{thm}\label{NLCLT}\upshape{(Lindeberg CLT for random $N$-vectors)}
{\em Suppose that $\left\{\Xi_{n,k}\right\}$ satisfies} Lindeberg's condition {\em in the sense that}
\begin{equation*}
\forall \epsilon > 0 : \sum_{k=1}^n \mathbb{E}\left[\left|\Xi_{n,k}\right|^2 ; \left|\Xi_{n,k}\right| > \epsilon\right] \rightarrow 0.\label{LindebergConditionN}
\end{equation*}
{\em Then}
\begin{equation*}
\Sigma_n \stackrel{w}{\rightarrow} \Xi.
\end{equation*}	
\end{thm}

It is customary to consider the distance 
\begin{equation*}
\sup_{A \in \mathcal{C}} \left|\mathbb{P}\left[\textrm{\upshape{H}} \in A\right] - \mathbb{P}\left[\textrm{\upshape{H}}^\prime \in A\right]\right|,
\end{equation*}
$\mathcal{C}$ being the collection of all convex Borel subsets of Euclidean $N$-space, between random $N$-vectors $\textrm{\upshape{H}}$ and $\textrm{\upshape{H}}^\prime$. Note that this distance is stronger than the earlier introduced Kolmogorov distance in the case where $N = 1$.

The question whether Theorem \ref{thm:BHall} can be extended to the multivariate setting is still open. However, multivariate versions of Stein's method (Barbour \cite{Ba}, G{\"o}tze \cite{G}, Goldstein and Rinott \cite{GR}, Chatterjee and Meckes \cite{CM}, Meckes \cite{M}, Reinert and R{\"o}llin \cite{RR}, Nourdin, Peccati and R{\'e}veillac \cite{NPR}) and of the Berry-Esseen inequality (G{\"o}tze \cite{G}, Rinott and Rotar \cite{RiRo}, Bentkus \cite{Be}, Bhattacharya and Holmes \cite{BhH}, Chen and Fang \cite{CF}) have been the object of extensive study. In this spirit, Chen and Fang have recently obtained the following result in \cite{CF}.

\begin{thm}\upshape{(Berry-Esseen inequality for random $N$-vectors)}
{\em There exists a universal constant $C > 0$ such that}
\begin{equation}
\sup_{A \in \mathcal{C}} \left|\mathbb{P}\left[\Xi \in A\right] - \mathbb{P}\left[\Sigma_n \in A\right]\right| \leq C \sqrt{N} \sum_{k=1}^n \mathbb{E}\left[\left|\Xi_{n,k}\right|^3\right].\label{NBE}
\end{equation}
\end{thm}

It was shown in \cite{CF} that the constant $C$ in (\ref{NBE}) can be taken equal to $115$. An issue of importance is the fact that the upper bound in (\ref{NBE}) is of order $O\left(\sqrt{N}\right)$, the sharpest obtained so far. We also notice that Bentkus has established in \cite{Be} an inequality of the type (\ref{NBE}) with an upper bound of order $O\left(\sqrt[4]{N}\right)$ under the additional assumption that $\Xi_{n,1}, \ldots, \Xi_{n,n}$ be identically distributed.

At this point it is natural to ask for a version of Theorem \ref{thm:BerckIneq} for random $N$-vectors, but, even with a multivariate version of Stein's method at hand, there seem to be some intrinsic obstructions towards obtaining such a result. However, if, in the spirit of e.g. \cite{GJT}, we consider $\phi_\Xi$ and $\phi_{\Sigma_n}$, where
\begin{equation*}
\phi_{\textrm{\upshape{H}}}(t) = \mathbb{E}\left[\exp\left(- i \left\langle t, \textrm{\upshape{H}} \right\rangle\right)\right],\quad t \in \mathbb{R}^N, 
\end{equation*}
represents the Fourier transform of the random $N$-vector $\textrm{\upshape{H}}$, instead of the cumulative distribution functions $F_{\Xi}$ and $F_{\Sigma_n}$, then we can show that Stein's method as outlined in e.g. \cite{M}, \cite{NPR} and \cite{CF} becomes applicable to get our main results, Theorem \ref{MainThm} and Corollary \ref{CorMainThm}. The latter is a quantitative multivariate Lindeberg CLT of the same taste as Theorem \ref{thm:BerckIneq}. The crux of the matter consists in deriving an explicit integral representation for the Hessian matrix of a solution to the Stein equation with test function $e_t(x) = \exp (- i \left\langle t,x\right\rangle)$, where $t,x \in \mathbb{R}^N$ (Proposition \ref{lem:GoodFormulaHessChar}).

\section{Formulation of the main results}

We keep the terminology and the notation of the previous section.

Let $\phi_{\textrm{\upshape{H}}}$ be the Fourier transform of the random $N$-vector H. That is, for $t,x \in \mathbb{R}^N$,
\begin{equation*}
\phi_{\textrm{\upshape{H}}}(t) = \mathbb{E}\left[e_t(\textrm{\upshape{H}})\right]
\end{equation*}
with
\begin{equation*}
e_t(x) = \exp(- i \left\langle t, x \right \rangle).
\end{equation*}

It is well-known that if we put $\rrr$ the space of all random $N$-vectors then
$$
(\phi_{\textrm{\upshape{H}}}: \rrr, \ttt_w \rightarrow \mathbb{R}, \ttt_{d_E}: H \mapsto \mathbb{E}\left[e_t(\textrm{\upshape{H}})\right])_{t \in \mathbb{R}^N}
$$
is an initial source, i.e.~the topology of weak convergence is the weakest topology making all these maps continuous. We now lift this canonical way to obtain the weak topology from the topological to the approach level. In order to do this we simply replace the Euclidean topology on $\mathbb{R}$ by the Euclidean metric and take the initial approach structure (here simply denoted by its limit operator $\lambda$) rather than the initial topology

$$
(\phi_{\textrm{\upshape{H}}}: \rrr, \lambda \rightarrow \mathbb{R}, d_E: H \mapsto \mathbb{E}\left[e_t(\textrm{\upshape{H}})\right])_{t \in \mathbb{R}^N}
$$

Note that an initial metric does not exist, it makes no sense to ask for a ``weakest'' metric, it is absolutely required to go to the realm of approach spaces in order to find a solution.

It then follows from general results in approach theory (see e.g.~\cite{L15}) that the limit operator in this space is given by the following formula, where $\Sigma_n, n \in \mathbb{N}$ and $\Xi$ are random variables:
\begin{equation*}
\lambda\left(\Sigma_n \rightarrow \Xi\right) = \sup_{t \in \mathbb{R}^N} \limsup_{n \rightarrow \infty} \left|\phi_\Xi(t) - \phi_{\Sigma_n}(t)\right|.\label{FLimitOperator}
\end{equation*}
We refer the reader interested in the fundamentals of approach theory to \cite{L97}, \cite{L15}, \cite{BLV11} and \cite{BLV11An}. For the sake of this paper, the following result, which reveals that the number $\lambda\left(\Sigma_n \rightarrow \Xi\right)$ indeed measures how far the sequence $(\Sigma_n)_n$ deviates from being weakly convergent to $\Xi$, suffices.

\begin{pro}\label{proLimOp}
\begin{equation}
0 \leq \lambda\left(\Sigma_n \rightarrow \Xi\right) \leq 2\label{limop}
\end{equation}
and 
\begin{equation}
\lambda\left(\Sigma_n \rightarrow \Xi\right) = 0 \Leftrightarrow \Sigma_n \stackrel{w}{\rightarrow} \Xi.\label{limop0}
\end{equation}
\end{pro}

\begin{proof}
(\ref{limop}) is trivial. (\ref{limop0}) follows from L{\'e}vy's Continuity Theorem, which states that weak convergence of random vectors is equivalent to pointwise convergence of their Fourier transforms.
\end{proof}

\begin{lem}\label{SumNSTA}
\begin{equation}
 \sum_{k = 1}^n \mathbb{E}\left[\left|\Xi_{n,k}\right|^2\right]  = N.\label{eq:SumNSTA}
\end{equation}
\end{lem}

\begin{proof}
The calculation
\begin{align*}
 \sum_{k = 1}^n \mathbb{E}\left[\left|\Xi_{n,k}\right|^2\right] & = \sum_{k = 1}^n \mathbb{E}\left[\sum_{l = 1}^N \Xi_{n,k,l}^2\right]\\
& = \sum_{l = 1}^N \mathbb{E}\left[\sum_{k = 1}^n \Xi_{n,k,l}^2\right]\\
& = \sum_{l = 1}^N \mathbb{E}\left[\left(\sum_{k = 1}^n \Xi_{n,k,l}\right)^2\right]\\
& = \sum_{l = 1}^N \textrm{\upshape{cov}}\left(\sum_{k = 1}^n \Xi_{n,k}\right)_{l,l}\\
& \quad (\textrm{\upshape{cov}}\left(\sum_{k = 1}^n \Xi_{n,k}\right) = I_{N \times N}) = N
\end{align*}
where $\Xi_{n,k} \textrm{ and } \Xi_{n,j} \textrm{ are independent if } k \neq j \textrm{ and } \mathbb{E}\left[\Xi_{n,k}\right] = 0 \textrm{ for all } k$ finishes the proof.
\end{proof}

We say that $\{\Xi_{n,k}\}$ is {\em infinitesimal} iff
\begin{equation*}
\forall \epsilon > 0 : \max_{k = 1}^n \mathbb{P}\left[\left|\Xi_{n,k}\right| > \epsilon\right] \rightarrow 0
\end{equation*}
and we extend the notion of {\em Lindeberg index} by putting
\begin{equation*}
\textrm{\upshape{Lin}}\left(\left\{\Xi_{n,k}\right\}\right) = \sup_{\epsilon > 0} \limsup_{n \rightarrow \infty} \sum_{k=1}^n \mathbb{E}\left[\left|\Xi_{n,k}\right|^2 ; \left|\Xi_{n,k}\right| > \epsilon\right].
\end{equation*}
It follows from Lemma \ref{SumNSTA} that $0 \leq \textrm{\upshape{Lin}}\left(\left\{\Xi_{n,k}\right\}\right) \leq N$ and it is clear that $\left\{\Xi_{n,k}\right\}$ satisfies Lindeberg's condition if and only if $\textrm{\upshape{Lin}}\left(\left\{\Xi_{n,k}\right\}\right) = 0$.

\begin{pro}\label{LinImpInf}
If $\left\{\Xi_{n,k}\right\}$ satisfies Lindeberg's condition, then it is infinitesimal.
\end{pro}

\begin{proof}
For $\epsilon > 0$, Chebyshev's Inequality gives
\begin{align*}
&\max_{k=1}^n \mathbb{P}\left[\left|\Xi_{n,k}\right| > \epsilon\right]\\
& \leq \epsilon^{-2} \max_{k=1}^n  \mathbb{E}\left[\left|\Xi_{n,k}\right|^2\right]\\
& = \epsilon^{-2} \max_{k=1}^n \mathbb{E}\left[\left|\Xi_{n,k}\right|^2 ; \left|\Xi_{n,k}\right| > \epsilon^2\right] + \epsilon^{-2} \max_{k=1}^n \mathbb{E}\left[\left|\Xi_{n,k}\right|^2 ; \left|\Xi_{n,k}\right| \leq \epsilon^2\right]\\
& \leq \epsilon^{-2} \sum_{k = 1}^n \mathbb{E}\left[\left|\Xi_{n,k}\right|^2 ; \left|\Xi_{n,k}\right| > \epsilon^2\right] + \epsilon^2
\end{align*}
from which the proposition easily follows.
\end{proof}

For an $N$-STA $\{\Eta_{n,k}\}$, we define the auxiliary number
\begin{equation*}
L(\{\Xi_{n,k}\},\{\Eta_{n,k}\}) = \sup_{t \in \mathbb{R}^N} \limsup_{n \rightarrow \infty} \sum_{k = 1}^n \mathbb{E}\left[\left|\Xi_{n,k}\right|^2 ; \left|\langle\Eta_{n,k},t\rangle\right| > 1\right].\label{numberL}
\end{equation*}
Proposition \ref{LProps} below shows how $L(\{\Xi_{n,k}\},\{\Eta_{n,k}\})$ is linked to both the Lindeberg index and the condition of being infinitesimal.

\begin{pro}\label{LProps}
\begin{equation}
0 \leq L(\{\Xi_{n,k}\},\{\Eta_{n,k}\}) \leq N\label{IneqsL}.
\end{equation}
Also,
\begin{equation}
L(\{\Xi_{n,k}\},\{\Xi_{n,k}\}) \leq \textrm{\upshape{Lin}}\left(\left\{\Xi_{n,k}\right\}\right)\label{IneqLLin} 
\end{equation}
and the inequality in (\ref{IneqLLin}) becomes an equality if $N = 1$.
Finally, let $\left\{\Xi_{n,k}^0\right\}$ be any independent copy of $\left\{\Xi_{n,k}\right\}$. Then
\begin{equation}
\{\Xi_{n,k}\} \textrm{ is infinitesimal } \Rightarrow  L\left(\left\{\Xi_{n,k}\right\}, \left\{\Xi_{n,k}^0\right\}\right) = 0.\label{infL}
\end{equation}
\end{pro}

\begin{proof}
(\ref{eq:SumNSTA}) entails (\ref{IneqsL}). Furthermore, by the Cauchy-Schwarz Inequality, for $t \in \mathbb{R}^N \setminus \left\{0\right\}$,
\begin{equation*}
\sum_{k = 1}^n \mathbb{E}\left[\left|\Xi_{n,k}\right|^2 ; \left|\left\langle\Xi_{n,k},t\right\rangle\right| > 1\right] \leq \sum_{k = 1}^n \mathbb{E}\left[\left|\Xi_{n,k}\right|^2 ; \left|\Xi_{n,k}\right| > \left|t\right|^{-1}\right]
\end{equation*}
proving (\ref{IneqLLin}). If $N = 1$, then the inequality in (\ref{IneqLLin}) trivially becomes an equality. Finally, suppose that $\left\{\Xi_{n,k}\right\}$ is infinitesimal and let $\left\{\Xi^0_{n,k}\right\}$ be an independent copy of $\left\{\Xi_{n,k}\right\}$. Then, by the Cauchy-Schwarz Inquality and (\ref{eq:SumNSTA}), for $t \in \mathbb{R}^N \setminus \left\{0\right\}$,
\begin{align*}
&\sum_{k = 1}^n \mathbb{E}\left[\left|\Xi_{n,k}\right|^2 ; \left|\left\langle\Xi_{n,k}^0,t\right\rangle\right| > 1\right]\\
& \leq \sum_{k = 1}^n \mathbb{E}\left[\left|\Xi_{n,k}\right|^2 \right] \mathbb{P}\left[\left|\Xi_{n,k}^0\right| > \left|t\right|^{-1}\right]\\
& \leq \max_{k = 1}^n \mathbb{P}\left[\left|\Xi_{n,k}^0\right| > \left|t\right|^{-1}\right] \sum_{k = 1}^n \mathbb{E}\left[\left|\Xi_{n,k}\right|^2\right]\\
& = N \max_{k=1}^n \mathbb{P}\left[\left|\Xi_{n,k}^0\right| > \left|t\right|^{-1}\right]
\end{align*}
which establishes (\ref{infL}).
\end{proof}

We are now in a position to state our main results. The proof of Theorem \ref{MainThm} is deferred to the next section.

 \begin{thm}\label{MainThm} Let $\left\{\Xi_{n,k}^0\right\}$ be an independent copy of $\left\{\Xi_{n,k}\right\}$. Then
 
 \begin{equation}
\lambda(\Sigma_n \rightarrow \Xi) \leq 2 \left(L\left(\left\{\Xi_{n,k}\right\},\left\{\Xi_{n,k}\right\}\right) + L\left(\left\{\Xi_{n,k}\right\},\left\{\Xi_{n,k}^0\right\}\right) \right).\label{BigIneq}
\end{equation}

In particular, for $t \in \mathbb{R}^N$,
\begin{align}
& \limsup_{n \rightarrow \infty} \left|\phi_{\Xi}(t) - \phi_{\Sigma_n}(t)\right| \label{MainForm} \\
& \leq 2 \left(1 - \exp\left(-\frac{1}{2} \left|t\right|^2\right)\right) \left(L\left(\left\{\Xi_{n,k}\right\},\left\{\Xi_{n,k}\right\}\right) + L\left(\left\{\Xi_{n,k}\right\},\left\{\Xi_{n,k}^0\right\}\right) \right).\nonumber
\end{align}
 \end{thm}
 
 Theorem \ref{MainThm} has the following corollary, which is a multivariate quantitative Lindeberg CLT of the same taste as Theorem \ref{thm:BerckIneq}. The proof of this theorem requires several steps and intermediate results, therefore we defer it to the third section.
 
 \begin{gev}\label{CorMainThm}{\em(Quantitative Lindeberg CLT for the Fourier transforms of random $N$-vectors)}
 Suppose that $\left\{\Xi_{n,k}\right\}$ is infinitesimal. Then
 \begin{equation}
 \lambda(\Sigma_n \rightarrow \Xi) \leq 2 \textrm{\upshape{Lin}}\left(\left\{\Xi_{n,k}\right\}\right).\label{AppMainIneq}
 \end{equation}
 More explicitly,
 \begin{equation}
 \sup_{t \in \mathbb{R}^N} \limsup_{n \rightarrow \infty} \left|\phi_{\Xi}(t) - \phi_{\Sigma_n}(t)\right| \leq 2 \textrm{\upshape{Lin}}\left(\left\{\Xi_{n,k}\right\}\right).\label{MainIneq}
 \end{equation}
 \end{gev}
 
 \begin{proof}
Recall that  Proposition \ref{LProps} entails $L\left(\left\{\Xi_{n,k}\right\},\left\{\Xi_{n,k}\right\}\right) \leq \textrm{\upshape{Lin}}\left(\left\{\Xi_{n,k}\right\}\right)$ and that, $\left\{\Xi_{n,k}\right\}$ being infinitesimal,  $L\left(\left\{\Xi_{n,k}\right\},\left\{\Xi_{n,k}^0\right\}\right) = 0$. Thus (\ref{BigIneq}) immediately gives (\ref{AppMainIneq}).
 \end{proof}
 
 {\bf Remarks} (1) Corollary \ref{CorMainThm} is stronger than Theorem \ref{NLCLT}. Indeed, suppose that $\left\{\Xi_{n,k}\right\}$ satisfies Lindeberg's condition. Then, by Proposition \ref{LinImpInf}, $\left\{\Xi_{n,k}\right\}$ is also infinitesimal. But then Corollary \ref{CorMainThm} implies that $\lambda(\Sigma_n \rightarrow \Xi) = 0$ and thus, by Proposition \ref{proLimOp}, $\Sigma_n \stackrel{w}{\rightarrow} \Xi$. The advantage of Theorem \ref{MainThm} is that it continues to be informative for STA's such as $\left\{\eta_{\alpha,n,k}\right\}$, defined by (\ref{deta0}), (\ref{deta1}) and (\ref{deta2}), for which Lindeberg's condition is not satisfied, whereas Theorem \ref{NLCLT} fails to be applicable for such STA's.
 
(2) For the large class of infinitesimal $N$-STA's, (\ref{MainIneq}) yields an upper bound which does not depend on the dimension $N$. This suggests the possibility of extending the result to an infinite dimensional setting. Such extensions will be discussed elsewhere.
 
(3) The left-hand side in (\ref{MainIneq}) is optimal in the sense that it is impossible to get similar upper bounds for $\limsup_{n \rightarrow \infty} \sup_{t \in \mathbb{R}^N} \left|\phi_{\Xi}(t) - \phi_{\Sigma_n}(t)\right|.$ Indeed, let $N = 1$ and consider i.i.d. random variables $\xi_1, \xi_2, \ldots$ with $\mathbb{P}\left[\xi_k = -1\right] = \mathbb{P}\left[\xi_k = 1\right] = 1/2$ and put
$\xi_{n,k} =  \xi_k / \sqrt{n}$ and $S_n = \sum_{k=1}^n \xi_{n,k}.$ Then $\left\{\xi_{n,k}\right\}$ is a 1-STA such that $\textrm{\upshape{Lin}}\left(\left\{\xi_{n,k}\right\}\right) = 0$, but, for each $n$, it holds that $\sup_{t \in \mathbb{R}} \left|\phi_{\xi}(t) - \phi_{S_n}(t)\right| = \sup_{t \in \mathbb{R}} \left|\exp\left(-t^2/2\right) - \cos^n\left(t/\sqrt{n}\right)\right| = 1.$

\section{Proof of Theorem \ref{MainThm}}

We keep the terminology and the notation of the previous sections.

The proof of Theorem \ref{MainThm} heavily depends on a multivariate version of Stein's method as outlined in e.g. \cite{M}, \cite{NPR} and \cite{CF}.

Let $h : \mathbb{R}^N \rightarrow \mathbb{C}$ be bounded and twice continuously differentiable with bounded first order and second order partial derivatives and let $f_h : \mathbb{R}^N \rightarrow \mathbb{C}$ be the solution to the Stein equation
\begin{equation}
\left\langle x , \nabla f(x) \right\rangle - \Delta f (x) = \mathbb{E}\left[h(\Xi)\right] - h(x)\label{eq:Stein}
\end{equation}
given by
\begin{equation}
 f_h(x) = - \int_0^1 \frac{1}{2s} \mathbb{E}\left[h(\Xi) - h\left(\sqrt{s} x + \sqrt{1 - s} \Xi\right)\right] ds,\label{eq:SteinSol}
\end{equation}
see \cite{M} or \cite{NPR}. Furthermore, let $\textrm{\upshape{Hess}} f_h (x)$ stand for the Hessian matrix of $f_h$ at $x$ and put
\begin{equation*}
D_{\textrm{\upshape{Hess}} f_h}(x,y) = \textrm{\upshape{Hess}}f_h(x) - \textrm{\upshape{Hess}} f_h(y). 
\end{equation*}
Finally, let $\left\{\Xi_{n,k}^0\right\}$ be an independent copy of $\left\{\Xi_{n,k}\right\}$. 

The following proposition follows from the explicit structure of the Stein equation.

\begin{pro}\label{lem:FirstLemma}
\begin{align}
& \mathbb{E}\left[h\left(\Xi\right) - h\left(\Sigma_n\right)\right]\label{eq:FirstLem}\\
&=\sum_{k=1}^n \int_0^1\mathbb{E}\left[\left\langle \Xi_{n,k},D_{\textrm{\upshape{Hess}} f_h}\left(\sum_{j \neq k} \Xi_{n,j} + r \Xi_{n,k},\sum_{j \neq k} \Xi_{n,j}\right) \Xi_{n,k}\right\rangle\right]dr\nonumber\\
& \quad - \sum_{k=1}^n\mathbb{E}\left[\left\langle\Xi_{n,k}, D_{\textrm{\upshape{Hess}} f_h}\left(\sum_{k = 1}^n \Xi_{n,k}^0,\sum_{j \neq k} \Xi_{n,j}^0\right)\Xi_{n,k}\right\rangle \right].\nonumber
\end{align}
\end{pro}
\begin{proof}
The fact that $f_h$ is a solution to the Stein equation (\ref{eq:Stein}) leads to
\begin{align}
& \mathbb{E}\left[h\left(\Xi\right) - h\left(\Sigma_n\right)\right]\nonumber\\  
& =\sum_{k=1}^n \mathbb{E}\left[\left\langle \Xi_{n,k}, \nabla f_h\left(\sum_{k=1}^n \Xi_{n,k}\right) \right \rangle- \Delta f_h \left(\sum_{k=1}^n \Xi_{n,k}\right)\right].\label{aeq1}
\end{align}
Furthermore,
\begin{align}
& \sum_{k=1}^n \mathbb{E}\left[\left\langle \Xi_{n,k}, \nabla f_h\left(\sum_{k=1}^n \Xi_{n,k}\right) \right \rangle- \Delta f_h \left(\sum_{k=1}^n \Xi_{n,k}\right)\right]\nonumber\\
&=\sum_{k=1}^n \mathbb{E}\bigg[\bigg\langle \Xi_{n,k},\label{aeq2}\\
& \nabla f_h\left(\sum_{k = 1}^n \Xi_{n,k}\right) - \nabla f_h \left(\sum_{j \neq k} \Xi_{n,j}\right) - \textrm{\upshape{Hess}} f_h \left(\sum_{j \neq k} \Xi_{n,j}\right) \Xi_{n,k}\bigg\rangle\bigg] \nonumber \\
& - \sum_{k=1}^n\mathbb{E}\bigg[\bigg\langle\Xi_{n,k},\left(\textrm{\upshape{Hess}} f_h\left(\sum_{k=1}^n \Xi_{n,k}^0\right) - \textrm{\upshape{Hess}} f_h \left(\sum_{j \neq k} \Xi_{n,j}^0\right)\right)\Xi_{n,k} \bigg\rangle \bigg]\nonumber
\end{align}
which is seen by calculating the right-hand side and noticing the following three facts. Firstly,
\begin{align*}
&\sum_{k=1}^n \mathbb{E}\left[\left\langle\Xi_{n,k},\nabla f_h \left(\sum_{j \neq k} \Xi_{n,j}\right)\right\rangle\right]\\
& = \sum_{k=1}^n \sum_{l=1}^N \mathbb{E}\left[\Xi_{n,k,l} \frac{\partial f_h}{\partial x_l} \left(\sum_{j \neq k} \Xi_{n,j}\right)\right]\\
& \quad \quad (\Xi_{n,k} \textrm{ and } \sum_{j \neq k} \Xi_{n,j} \textrm{ are independent})\\
& = \sum_{k=1}^n \sum_{l=1}^N \mathbb{E}\left[\Xi_{n,k,l}\right] \mathbb{E}\left[\frac{\partial f_h}{\partial x_l} \left(\sum_{j \neq k} \Xi_{n,j}\right)\right]\\
&\quad \quad (\mathbb{E}\left[\Xi_{n,k}\right] = 0)\\
& = 0.
\end{align*}
Secondly,
\begin{align*}
& \sum_{k=1}^n \mathbb{E}\left[\left\langle\Xi_{n,k},\textrm{\upshape{Hess}} f_h\left(\sum_{j \neq k} \Xi_{n,j}\right) \Xi_{n,k}\right\rangle\right]\\
& =  \sum_{k=1}^n \sum_{l=1}^N \sum_{m =1}^N \mathbb{E}\left[\frac{\partial^2 f_h}{\partial x_l \partial x_m} \left(\sum_{j \neq k} \Xi_{n,j}\right) \Xi_{n,k,l} \Xi_{n,k,m}\right]\\
& \quad \quad (\Xi_{n,k} \textrm{ and } \sum_{j \neq k} \Xi_{n,j} \textrm{ are independent})\\
& = \sum_{k=1}^n \sum_{l=1}^N \sum_{m =1}^N \mathbb{E}\left[\frac{\partial^2 f_h}{\partial x_l \partial x_m} \left(\sum_{j \neq k} \Xi_{n,j}\right)\right] \textrm{\upshape{cov}}\left(\Xi_{n,k}\right)_{l,m}\\
& \quad \quad (\left\{\Xi^0_{n,k}\right\} \textrm{ is an independent  copy of } \left\{\Xi_{n,k}\right\})\\
& = \sum_{k=1}^n \sum_{l=1}^N \sum_{m =1}^N \mathbb{E}\left[\frac{\partial^2 f_h}{\partial x_l \partial x_m} \left(\sum_{j \neq k} \Xi^0_{n,j}\right)\right] \textrm{\upshape{cov}}\left(\Xi_{n,k}\right)_{l,m}\\
& \quad \quad (\Xi^0_{n,k} \textrm{ and } \sum_{j \neq k} \Xi_{n,j} \textrm{ are independent})\\ 
&  =  \sum_{k=1}^n \sum_{l=1}^N \sum_{m =1}^N \mathbb{E}\left[\frac{\partial^2 f_h}{\partial x_l \partial x_m} \left(\sum_{j \neq k} \Xi^0_{n,j}\right) \Xi_{n,k,l} \Xi_{n,k,m}\right]\\
& = \sum_{k=1}^n \mathbb{E}\left[\left\langle\Xi_{n,k},\textrm{\upshape{Hess}} f_h\left(\sum_{j \neq k} \Xi^0_{n,j}\right) \Xi_{n,k}\right\rangle\right].\\
\end{align*}
Thirdly,
\begin{align*}
& \sum_{k=1}^n \mathbb{E}\left[\left\langle \Xi_{n,k} , \textrm{\upshape{Hess}} f_h \left(\sum_{k=1}^n \Xi_{n,k}^0\right) \Xi_{n,k}\right\rangle\right]\\
& = \sum_{k=1}^n \sum_{l = 1}^N \sum_{m = 1}^N \mathbb{E}\left[\frac{\partial^2 f_h}{\partial x_l \partial x_m} \left(\sum_{k=1}^n \Xi_{n,k}^0\right) \Xi_{n,k,l} \Xi_{n,k,m}\right]\\
& \quad \quad (\left\{\Xi_{n,k}^0\right\} \textrm{ and } \left\{\Xi_{n,k}\right\} \textrm{ are independent})\\
& = \sum_{k=1}^n  \sum_{l = 1}^N \sum_{m = 1}^N \mathbb{E}\left[\frac{\partial^2 f_h}{\partial x_l \partial x_m} \left(\sum_{k=1}^n \Xi_{n,k}^0\right)\right] \textrm{\upshape{cov}}\left(\Xi_{n,k}\right)_{l,m}\\
& = \sum_{l=1}^N \sum_{m = 1}^N \mathbb{E}\left[\frac{\partial^2 f_h}{\partial x_l \partial x_m} \left(\sum_{k=1}^n \Xi_{n,k}^0\right)\right] \left(\sum_{k=1}^n \textrm{\upshape{cov}}\left(\Xi_{n,k}\right)_{l,m}\right)\\
& \quad \quad (\Xi_{n,1}, \ldots, \Xi_{n,n} \textrm{ are independent and } \mathbb{E}\left[\Xi_{n,k}\right] = 0)\\
& = \sum_{l=1}^N \sum_{m = 1}^N \mathbb{E}\left[\frac{\partial^2 f_h}{\partial x_l \partial x_m} \left(\sum_{k=1}^n \Xi_{n,k}^0\right)\right] \left(\textrm{\upshape{cov}}\left(\sum_{k=1}^n \Xi_{n,k}\right)_{l,m}\right)\\
& \quad \quad (\textrm{\upshape{cov}}\left(\sum_{k=1}^n \Xi_{n,k}\right) = I_{N \times N})\\
& = \mathbb{E}\left[\Delta f_h \left(\sum_{k=1}^n \Xi_{n,k}^0\right)\right]\\
& \quad \quad (\left\{\Xi_{n,k}^0\right\} \textrm{ is a copy of } \left\{\Xi_{n,k}\right\})\\
& = \mathbb{E}\left[\Delta f_h \left(\sum_{k=1}^n \Xi_{n,k}\right)\right].
\end{align*}
Finally, the Fundamental Theorem of Calculus reveals that
\begin{align}
&\nabla f_h \left(\sum_{k=1}^n \Xi_{n,k}\right) - \nabla f_h \left(\sum_{j \neq k} \Xi_{n,k}\right)\nonumber\\
& = \int_0^1 \textrm{\upshape{Hess}} f_h\left(\sum_{j \neq k} \Xi_{n,j} + r \Xi_{n,k}\right) \Xi_{n,k} dr.\label{aeq3}
\end{align}
Combining (\ref{aeq1}), (\ref{aeq2}) and (\ref{aeq3}) proves (\ref{eq:FirstLem}) and we are done.
\end{proof}

Proposition \ref{lem:FirstLemma} highlights the role of the Hessian matrix of $f_h$ in the search for an upper bound for expressions of the type $\left|\mathbb{E}\left[h\left(\Xi\right) - h\left(\Sigma_n\right)\right]\right|$. In the following proposition we establish an explicit integral representation for $\textrm{\upshape{Hess}} f_h $. We consider an $N$-vector $z \in \mathbb{C}^N$ as a $1 \times N$-matrix and we denote its transpose as $z^\tau$.

 \begin{pro}
 \begin{equation}
 \nabla f_h(x) = - \int_0^1 \frac{1}{2 \sqrt{s (1-s)}} \mathbb{E}\left[h\left(\sqrt{s} x + \sqrt{1 -s} \Xi\right) \Xi\right] ds\label{eq:GradSteinSol}
 \end{equation}
 and
 \begin{equation}
 \textrm{\upshape{Hess}}f_h(x) = - \int_0^1 \frac{1}{2(1-s)} \mathbb{E}\left[h\left(\sqrt{s}x + \sqrt{1-s} \Xi\right)\left(\Xi \Xi^\tau - I_{N \times N} \right)\right] ds.\label{eq:HessSteinSol}
 \end{equation}
 \end{pro}
 
 \begin{proof}
Using (\ref{eq:SteinSol}) and performing an integration by parts on the Gaussian expectation gives 
 \begin{equation*}
 \frac{\partial f_h}{\partial x_l} (a) = \int_0^1 \frac{1}{2 \sqrt{s (1-s)}} \mathbb{E}\left[h\left(\sqrt{s} a + \sqrt{1 -s} \Xi\right) \Xi_l\right] ds 
 \end{equation*}
 and (\ref{eq:GradSteinSol}) follows. Using (\ref{eq:GradSteinSol}) and again performing an integration by parts on the Gaussian expectation gives 
 \begin{equation*}
 \frac{\partial^2 f_h}{\partial x_l \partial x_m} (a)=  \int_0^1 \frac{1}{2(1-s)} \mathbb{E}\left[h\left(\sqrt{s}a + \sqrt{1-s} \Xi\right)\left(\Xi_l \Xi_m - \delta_{lm}\right)\right] ds,
 \end{equation*}
 with $\delta_{lm}$ the Kronecker delta, and (\ref{eq:HessSteinSol}) follows.
 \end{proof}
 
The singularity at $1$ of the integrand in (\ref{eq:HessSteinSol}) makes it hard to control $\textrm{\upshape{Hess}}f_h(x)$ for general $h$. However, we establish in Proposition \ref{lem:GoodFormulaHessChar} that for the specific choice 
 \begin{equation*}
 h(x) = e_t(x) = \exp\left(- i \left\langle t,x\right\rangle\right), \quad t \in \mathbb{R}^N, 
\end{equation*} 
the integral representation of $\textrm{\upshape{Hess}}f_h(x)$ does not contain the factor $\frac{1}{1-s}$ anymore. We first need two lemmas.
 
 \begin{lem}
Fix $y, t \in \mathbb{R}^N$ and $s \in \left[0,1\right]$. Put
\begin{equation*}
\alpha_{y,t,s} = y + i \sqrt{1 - s} t.
\end{equation*}
Then 
\begin{equation}
- i \sqrt{1 - s} \left\langle t, y \right\rangle - \frac{1}{2} \left|y\right|^2 = - \frac{1}{2} (1-s) \left|t\right|^2 - \frac{1}{2} \alpha_{y,t,s}^\tau \alpha_{y,t,s}.\label{eq:FirstForm}
\end{equation}
Furthermore,
\begin{equation}
y y^\tau - I_{N \times N} = \alpha_{y,t,s} \alpha_{y,t,s}^\tau - i \sqrt{1 - s} t \alpha_{y,t,s}^\tau - i \sqrt{1-s} \alpha_{y,t,s} t^\tau - (1 - s) t t^\tau - I_{N \times N}.\label{eq:SecondForm}
\end{equation}
\end{lem}

\begin{proof}
This is elementary.
\end{proof}

\begin{lem}
\begin{align}
& \mathbb{E}\left[e_t\left(\sqrt{s} x + \sqrt{1 - s} \Xi\right) \left(\Xi \Xi^\tau - I_{N \times N}\right)\right]\nonumber\\
& = - (1 - s) tt^\tau \exp\left(- i \sqrt{s} \left \langle t, x \right\rangle- \frac{1}{2} (1 - s) \left|t\right|^2\right).\label{eq:FirstBigForm}
\end{align}
\end{lem}

\begin{proof}
Put
\allowdisplaybreaks\begin{equation*}
\alpha_{y,t,s} = y + i \sqrt{1 - s} t.
\end{equation*}
From (\ref{eq:FirstForm}) and (\ref{eq:SecondForm}) we learn that 
\begin{align*}
& \mathbb{E}\left[e_t\left(\sqrt{s} x + \sqrt{1 - s} \Xi\right) \left(\Xi \Xi^\tau - I_{N \times N}\right)\right]\\
& = \frac{1}{(2\pi)^{N/2}} \int_{\mathbb{R}^N} \exp\left(- i \sqrt{s} \left\langle t , x\right\rangle - i \sqrt{1 - s} \left\langle t, y \right\rangle - \frac{1}{2} \left|y\right|^2\right) \left(y y^\tau - I_{N \times N}\right) dy\\
& = \exp\left(- i \sqrt{s} \left\langle t , x \right\rangle - \frac{1}{2} (1 - s) \left|t\right|^2\right)\\
&\quad \bigg( \frac{1}{(2 \pi)^{N/2}} \int_{\mathbb{R}^N} \left(\alpha_{y,t,s} \alpha_{y,t,s}^\tau - I_{N \times N}\right)\exp\left(-\frac{1}{2} \alpha_{y,t,s}^\tau \alpha_{y,t,s}\right) dy\\
&\quad - i \sqrt{1 - s} \frac{1}{\left(2\pi\right)^{N/2}} \int_{\mathbb{R}^N} t\alpha_{y,t,s}^\tau \exp\left(-\frac{1}{2}\alpha_{y,t,s}^\tau \alpha_{y,t,s}\right) dy\\
&\quad - i \sqrt{1 - s} \frac{1}{\left(2\pi\right)^{N/2}} \int_{\mathbb{R}^N} \alpha_{y,t,s} t^\tau \exp\left(-\frac{1}{2}\alpha_{y,t,s}^\tau \alpha_{y,t,s}\right) dy\\
&\quad - \left((1-s) tt^\tau + I_{N \times N} \right)\frac{1}{\left(2\pi\right)^{N/2}} \int_{\mathbb{R}^N} \exp\left(- \frac{1}{2} \alpha_{y,t,s}^\tau \alpha_{y,t,s}\right) dy\bigg)
\end{align*}
which, by Cauchy's Integral Theorem,
\begin{align*}
& = \exp\left(- i \sqrt{s} \left\langle t , x \right\rangle - \frac{1}{2} (1 - s) \left|t\right|^2\right) \bigg(\frac{1}{\left(2\pi\right)^{N/2}} \int_{\mathbb{R}^N} (y y^\tau \exp\left(-\frac{1}{2} \left|y\right|^2\right) dy\\
&\quad - i \sqrt{1 - s} \frac{1}{\left(2\pi\right)^{N/2}} \int_{\mathbb{R}^N} t y^\tau \exp\left(-\frac{1}{2} \left|y\right|^2\right) dy\\
&\quad - i \sqrt{1 - s} \frac{1}{\left(2\pi\right)^{N/2}} \int_{\mathbb{R}^N} y t^\tau \exp\left(-\frac{1}{2}\left|y\right|^2\right) dy\\
&\quad - \left((1-s) tt^\tau + I_{N \times N}\right) \frac{1}{\left(2 \pi\right)^{N/2}} \int_{\mathbb{R}^N} \exp\left(- \frac{1}{2} \left|y\right|^2\right) dy\bigg)\\
& = \exp\left(- i \sqrt{s} \left\langle t , x \right\rangle - \frac{1}{2} (1 - s) \left|t\right|^2\right)\\
& \quad \left( \textrm{\upshape{cov}}(\Xi) - i \sqrt{1-s} t \mathbb{E}\left[\Xi\right]^\tau - i \sqrt{1 -s} \mathbb{E}\left[\Xi\right] t ^\tau - (1 - s) t t^\tau - I_{N \times N}\right)
\end{align*}
and (\ref{eq:FirstBigForm}) follows.
\end{proof}

\begin{pro}\label{lem:GoodFormulaHessChar}
\begin{equation}
\textrm{\upshape{Hess}}f_{e_t}(x)  = - \frac{1}{2} tt^\tau \int_0^1 \exp\left(-i \sqrt{s} \left\langle t,x\right\rangle - \frac{1}{2} (1 - s) \left|t\right|^2\right) ds\label{eq:FirstBigBigForm}.
\end{equation}
In particular,
\begin{align}
& D_{\textrm{\upshape{Hess}} f_{e_t}}(x,y) \nonumber\\
& = - \frac{1}{2} t t^\tau \int_0^1 e_{\sqrt{s} t}(y) \left[e_{\sqrt{s}t}(x - y) - 1\right] \exp\left(-\frac{1}{2} (1 - s) \left|t\right|^2\right) ds.\label{eq:FormD}
\end{align}
\end{pro}

\begin{proof}
Combining (\ref{eq:HessSteinSol}) and (\ref{eq:FirstBigForm}) gives (\ref{eq:FirstBigBigForm}). Also, (\ref{eq:FormD}) follows immediately from (\ref{eq:FirstBigBigForm}).
\end{proof}
 
Proposition \ref{lem:FirstLemma} and Proposition \ref{lem:GoodFormulaHessChar} lead to the following result, which contains an explicit formula for the quantity $\mathbb{E}\left[e_t\left(\Xi\right) - e_t\left(\Sigma_n\right)\right]$ without any reference to the Stein equation.

\begin{pro}\label{JanThm}
\begin{align}
& \mathbb{E}\left[e_t(\Xi) - e_t(\Sigma_n)\right]\label{FormJan}\\
& = - \frac{1}{2} \int_0^1 \bigg( \int_0^1 \sum_{k = 1}^n \mathbb{E}\left[e_{\sqrt{s} t} \left(\sum_{j \neq k} \Xi_{n,j}\right) \left[e_{\sqrt{s} t}\left(r \Xi_{n,k}\right) - 1\right] \left|\left\langle\Xi_{n,k},t\right\rangle\right|^2\right] dr\nonumber\\
&\quad -  \sum_{k = 1}^n \mathbb{E}\left[e_{\sqrt{s} t} \left(\sum_{j \neq k} \Xi^0_{n,j}\right) \left[e_{\sqrt{s} t}\left(\Xi^0_{n,k}\right) - 1\right] \left|\left\langle\Xi_{n,k},t\right\rangle\right|^2\right] \bigg) e^{- \frac{1}{2} (1 -s) \left|t\right|^2} ds.\nonumber
\end{align}
\end{pro}

\begin{proof}
Applying (\ref{eq:FirstLem}) gives
\begin{align*}
& \mathbb{E}\left[e_t\left(\Xi\right) - e_t\left(\Sigma_n\right)\right]\\
& = \sum_{k=1}^n \int_0^1\mathbb{E}\left[\left\langle \Xi_{n,k},D_{\textrm{\upshape{Hess}} f_{e_t}}\left(\sum_{j \neq k} \Xi_{n,j} + r \Xi_{n,k},\sum_{j \neq k} \Xi_{n,j}\right) \Xi_{n,k}\right\rangle\right]dr\\
& \quad -  \sum_{k=1}^n\mathbb{E}\left[\left\langle\Xi_{n,k}, D_{\textrm{\upshape{Hess}} f_{e_t}}\left(\sum_{k = 1}^n \Xi_{n,k}^0,\sum_{j \neq k} \Xi_{n,j}^0\right)\Xi_{n,k}\right\rangle \right]
\end{align*}
which, using (\ref{eq:FormD}) and the elementary equality $\left\langle x,tt^\tau x\right\rangle = \left|\left\langle x,t\right\rangle\right|^2$,
\begin{align*}
& = \frac{1}{2} \int_0^1 \bigg( \int_0^1 \sum_{k = 1}^n \mathbb{E}\left[e_{\sqrt{s} t} \left(\sum_{j \neq k} \Xi_{n,j}\right) \left[e_{\sqrt{s} t}\left(r \Xi_{n,k}\right) - 1\right] \left|\left\langle\Xi_{n,k},t\right\rangle\right|^2\right] dr\\
&\quad -  \sum_{k = 1}^n \mathbb{E}\left[e_{\sqrt{s} t} \left(\sum_{j \neq k} \Xi^0_{n,j}\right) \left[e_{\sqrt{s} t}\left(\Xi^0_{n,k}\right) - 1\right] \left|\left\langle\Xi_{n,k},t\right\rangle\right|^2\right] \bigg) e^{- \frac{1}{2} (1 -s) \left|t\right|^2} ds
\end{align*}
proving the desired formula.
\end{proof}

Proposition \ref{JanThm} is crucial for the proof of Theorem \ref{MainThm}. We need one more lemma.

\begin{lem}\label{lem:LastLem}
Fix $t \in \mathbb{R}^N$, $r,s \in \left[0,1\right]$ and $\epsilon > 0$. Then
\begin{equation}
 \sum_{k = 1}^n\mathbb{E}\left[\left|e_{\sqrt{s} t} (r \Xi_{n,k}) - 1\right|\left|\Xi_{n,k}\right|^2\right]
 \leq \epsilon N + 2 \sum_{k = 1}^n \mathbb{E}\left[\left|\Xi_{n,k}\right|^2 ; \left|\left\langle\Xi_{n,k},t\right\rangle\right| > \epsilon\right] \label{eq:LastLemEq1}
\end{equation} 
and
\begin{equation}
 \sum_{k = 1}^n\mathbb{E}\left[\left|e_{\sqrt{s} t} (\Xi^0_{n,k}) - 1\right| \left|\Xi_{n,k}\right|^2\right]
\leq \epsilon N + 2 \sum_{k = 1}^n \mathbb{E}\left[\left|\Xi_{n,k}\right|^2 ; \left|\left\langle\Xi_{n,k}^0,t\right\rangle\right| > \epsilon\right].\label{eq:LastLemEq2}
\end{equation}
\end{lem}

\begin{proof}
The calculation
\begin{align*}
&\sum_{k=1}^n \mathbb{E}\left[\left|e_{\sqrt{s} t} \left(r \Xi_{n,k}\right) - 1\right|\left|\Xi_{n,k}\right|^2\right]\\
& = \sum_{k = 1}^n \mathbb{E}\left[\left|\exp\left(- i \sqrt{s} r \left\langle t ,\Xi_{n,k}\right\rangle\right) - 1\right|\left|\Xi_{n,k}\right|^2 ; \left|\left\langle\Xi_{n,k},t\right\rangle\right| \leq \epsilon\right]\\
& \quad + \sum_{k=1}^n \mathbb{E}\left[\left|\exp\left(- i \sqrt{s} r \left\langle t ,\Xi_{n,k}\right\rangle\right) - 1\right| \left|\Xi_{n,k}\right|^2 ; \left|\left\langle\Xi_{n,k},t\right\rangle\right| > \epsilon\right]\\
& \leq \epsilon \sum_{k = 1}^n \mathbb{E}\left[\left|\Xi_{n,k}\right|^2\right] + 2 \sum_{k=1}^n \mathbb{E}\left[\left|\Xi_{n,k}\right|^2 ; \left|\left\langle\Xi_{n,k},t\right\rangle\right| > \epsilon\right]\\
& \quad (\textrm{Lemma \ref{SumNSTA}})\\
& = \epsilon N + 2 \sum_{k = 1}^n \mathbb{E}\left[\left|\Xi_{n,k}\right|^2 ; \left|\left\langle\Xi_{n,k},t\right\rangle\right| > \epsilon\right]
\end{align*}
proves (\ref{eq:LastLemEq1}). The proof of (\ref{eq:LastLemEq2}) is similar.
\end{proof}

\begin{proof} [Proof of Theorem \ref{MainThm}]
Applying (\ref{FormJan}) gives
\begin{align*}
& \left|\phi_{\Xi}(t) - \phi_{\Sigma_n}(t)\right|\\
& = \left|\mathbb{E}\left[e_t\left(\Xi\right) - e_t\left(\Sigma_n\right)\right]\right|\\
& = \bigg| \frac{1}{2} \int_0^1 \bigg( \int_0^1 \sum_{k = 1}^n \mathbb{E}\left[e_{\sqrt{s} t} \left(\sum_{j \neq k} \Xi_{n,j}\right) \left[e_{\sqrt{s} t}\left(r \Xi_{n,k}\right) - 1\right] \left|\left\langle\Xi_{n,k},t\right\rangle\right|^2\right] dr\\
&\quad -  \sum_{k = 1}^n \mathbb{E}\left[e_{\sqrt{s} t} \left(\sum_{j \neq k} \Xi^0_{n,j}\right) \left[e_{\sqrt{s} t}\left(\Xi^0_{n,k}\right) - 1\right] \left|\left\langle\Xi_{n,k},t\right\rangle\right|^2\right] \bigg) e^{- \frac{1}{2} (1 -s) \left|t\right|^2} ds\bigg|
\end{align*}
and this is, by the Cauchy-Schwarz Inequality,
\begin{align*}
& \leq \frac{1}{2} \left|t\right|^2 \int_0^1 \bigg(\int_0^1\sum_{k=1}^n \mathbb{E}\left[\left|e_{\sqrt{s} t} \left(r \Xi_{n,k}\right) - 1\right|\left|\Xi_{n,k}\right|^2\right]dr\\
& \quad + \sum_{k = 1}^n\mathbb{E}\left[\left|e_{\sqrt{s} t} (\Xi^0_{n,k}) - 1\right| \left|\Xi_{n,k}\right|^2\right]\bigg) e^{- \frac{1}{2} (1 -s) \left|t\right|^2} ds\\
\end{align*}
which, by (\ref{eq:LastLemEq1}) and (\ref{eq:LastLemEq2}), for any $\epsilon > 0$,
\begin{align*}
& \leq 2 \epsilon N + 2 \bigg(\sum_{k = 1}^n \mathbb{E}\left[\left|\Xi_{n,k}\right|^2 ; \left|\left\langle\Xi_{n,k},t\right\rangle\right| > \epsilon\right] \\
& + \sum_{k = 1}^n \mathbb{E}\left[\left|\Xi_{n,k}\right|^2 ; \left|\left\langle\Xi^0_{n,k},t\right\rangle\right| > \epsilon\right]\bigg)\frac{1}{2} \left|t\right|^2 \int_0^1 e^{- \frac{1}{2} (1-s) \left|t\right|^2} ds\\
&\quad \quad (\textrm{perform the change of variables } u = \frac{1}{2} (1-s) \left|t\right|^2)\\\
& = 2 \epsilon N + 2 \bigg(\sum_{k = 1}^n \mathbb{E}\left[\left|\Xi_{n,k}\right|^2 ; \left|\left\langle\Xi_{n,k},t\right\rangle\right| > \epsilon\right] \\
& + \sum_{k = 1}^n \mathbb{E}\left[\left|\Xi_{n,k}\right|^2 ; \left|\left\langle\Xi^0_{n,k},t\right\rangle\right| > \epsilon\right]\bigg)\left(1 - \exp\left(-\frac{1}{2} \left|t\right|^2\right)\right).\\
\end{align*}
Since
\begin{align*}
& \sup_{\epsilon > 0} \sup_{t \in \mathbb{R}^N} \limsup_{n \rightarrow \infty} \sum_{k = 1}^n \mathbb{E}\left[\left|\Xi_{n,k}\right|^2 ; \left|\left\langle\Eta_{n,k},t\right\rangle\right| > \epsilon\right]\\
&= \sup_{t \in \mathbb{R}^N} \limsup_{n \rightarrow \infty} \sum_{k = 1}^n \mathbb{E}\left[\left|\Xi_{n,k}\right|^2 ; \left|\left\langle\Eta_{n,k},t\right\rangle\right| > 1\right]\\
&= L\left(\left\{\Xi_{n,k}\right\},\left\{\Eta_{n,k}\right\}\right),
\end{align*}
the previous calculation establishes (\ref{MainForm}), finishing the proof of Theorem \ref{MainThm}.
\end{proof}

\end{document}